\documentclass[12pt]{amsart}
\usepackage{graphicx}
\usepackage{amssymb}
\usepackage{stmaryrd}
\usepackage{enumitem}

\usepackage{geometry}
\usepackage{tikz}
\usepackage{float}
\usepackage{url}
\usepackage{mathtools}
\usepackage{amsmath,amsfonts,amsthm}
\usepackage{mathrsfs} 
\newcommand{\irr}{\operatorname{irr}}
\newcommand{\case}[1]{\paragraph*{Case #1:}}
\newtheorem{hypothesize}{Hypothesize}
\DeclarePairedDelimiter\ceil{\lceil}{\rceil}
\DeclarePairedDelimiter\floor{\lfloor}{\rfloor}

\newtheorem{theorem}{Theorem}[section]
\newtheorem{lemma}[theorem]{Lemma}
\newtheorem{proposition}{Proposition}[section]

\newtheorem{definition}{Definition}
\newtheorem{example}{Example}

\author{Jasem Hamoud}
\address{\textbf{Jasem Hamoud} 
Department of Discrete Mathematics, Moscow Institute of Physics and Technology
}
\email{khamud@phystech.edu}
\thanks{}

\author{Alexei Belov-Kanel}
\address{\textbf{Alexei Belov-Kanel} 
Department of Discrete Mathematics, Moscow Institute of Physics and Technology, Dolgoprudnyi, Institutskiy Pereulok,
141700 Moscow, Russia
}
\email{kanelster@gmail.com}
\thanks{}

\author{Duaa Abdullah}
\address{\textbf{Duaa Abdullah:} Department of Discrete Mathematics, Moscow Institute of Physics and Technology }
\email{abdulla.d@phystech.edu}
\thanks{}

\title{On Topological Indices in Trees: Fibonacci Degree Sequences and Bounds}

\date{}
\begin{document}
\begin{abstract}
In this paper, we have studied bounds based on topological indicators, from which we selected Albertson index $\mathrm{irr}$ and the Sigma index $\sigma$. The Sigma index was defined through the following relationship:
\[
\sigma(G)=\sum_{uv\in E(G)}\left( d_u(G)-d_v(G) \right)^2.
\]
 We establish a precise formula for the Albertson index of a tree $T$ of order $n$ with a Fibonacci degree sequence $\mathscr{D} = (F_3, \dots, F_n)$. Additionally, we derive bounds for the minimum and maximum Albertson indices ($\irr_{\min}$ and $\irr_{\max}$) across various tree structures. Propositions and lemmas provide upper and lower bounds, incorporating parameters such as the maximum degree $ \Delta$, minimum degree $\delta$. We further relate the Albertson index to the second Zagreb index $M_2(T)$ and the forgotten index $F(T)$, establishing a new upper bound.

\end{abstract}

\maketitle

\noindent\rule{12.7cm}{1.0pt}


\medskip

\noindent

\medskip


\noindent\rule{12.7cm}{1.0pt}

\section{Introduction}\label{sec1}

Throughout this paper. Let $G=(V,E)$ be a simple, connected graph, where $n=|V(G)|$, $m=|E(G)|$,  the minimum and the maximum degree of $G$ are denoted by $\delta, \Delta$ where $1\leqslant \delta <\Delta \leq n-1$. Let $\mathcal{P}_n$ be a path of order $n$ and $\mathcal{S}_n$ the star tree of order $n$, typically denoted by $K_{1,n-1}$. Two graphs $G$ and $H$ are deemed identical if their vertex sets are equal, $V(G) = V(H)$, and their edge sets are equal, $E(G) = E(H)$, or equivalently, if $G \subseteq H$ and $H \subseteq G$. The graphs $G$ and $H$ are said to be isomorphic~\cite{HaynesHedetniemiSlater1998}, denoted $G \cong H$, if there exists a bijection $\psi : V(G) \to V(H)$ such that an edge $uv \in E(G)$ exists if and only if the edge $\psi(u)\psi(v) \in E(H)$. A subset $S \subseteq V$ of vertices in a graph $G = (V, E)$ is termed a dominating set~\cite{HaynesHedetniemiSlater1998} if every vertex $v \in V$ either belongs to $S$ or is adjacent to at least one vertex in $S$.
Let $e=u v \in E(G)$ an edge, then the imbalance of $e$ known as $\left|d_C(u)-d_G(v)\right|$ it is employee in ``Albertson Index''.  Dorjsembe,S., et al. in~\cite{Dorjsembe2023GutmanLI} for a path $u_0 u_1 \cdots u_t$ in graph $G$, then $imb_G\left(u_0, u_t\right)=\sum_{i=0}^{t-1}\left|d_G\left(u_i\right)-d_G\left(u_{i+1}\right)\right|$. In 1997, Albertson in~\cite{Albertson1997} mention to the imbalance of an edge $uv$ by $imb(uv)$ where we considered a graph $G$ is regular if all of its vertices have the exact same degree, then the irregularity measure defined in~\cite{Albertson1997, GutmanHansenMelot2005, AbdoBrandtDimitrov,Bell1992} as: 
\begin{equation}~\label{eqsec1n2}
    \operatorname{irr}(G)=\sum_{uv\in E(G)}\lvert d_u(G)-d_v(G) \rvert.
\end{equation}
Denote by $\irr_{\max}$ the maximum of Albertson index and $\irr_{\min}$ the minimum of Albertson index.  Ali, A. et al in~\cite{AliDimitrovEtAl2025} mention to total irregularity of $\mathscr{D}(G)=(d_1,d_2,\dots,d_i)$ a degree sequence of $G$ where $d_1\ge d_2\ge\cdots\ge d_n$, it defined in Definition~\ref{degreeseq}, then we have: 
\begin{equation}~\label{eqsec1n3}
    \operatorname{irr}(G)=2(n+1)m - 2 \sum_{i=1}^n id_i.
\end{equation}
\begin{definition}[Degree Sequence~\cite{AshrafiGhalavand,ZhangZhangGrayWang2013,LinaZhoubMiaob2024,Andriantiana2013Wagner,MolloyReed1995}]~\label{degreeseq}
Let $G=(V,E)$ be a simple graph, where $V=\{v_{1}, v_{2}, \ldots, v_{n}\}$, let $\mathscr{D}(G)=\left(d_{G}\left(v_{1}\right), d_{G}\left(v_{2}\right), \ldots, d_{G}\left(v_{n}\right)\right)$ be a degree sequence of $G$ where $d_{G}\left(v_{1}\right) \geqslant d_{G}\left(v_{2}\right) \geqslant \cdots \geqslant d_{G}\left(v_{n}\right)$. When $\mathscr{D}(G)=(k, k, \ldots, k)$, then $G$ is regular of degree $k$. Otherwise, the graph is irregular
\end{definition}
Let $\mathscr{D}=(d_1,\dots,d_n)$ be a non-increasing degree sequence with $n\geq 2$. There are a graph $G_1$ where $V(G_1)=\{v_2,\dots,v_n\}$, then, the degree vertices~\cite{Chartrand2012Zhang} are
\begin{equation}~\label{eqGraphican1}
    \deg_{G_1} v_i = 
\begin{cases}
    d_i - 1 & \text{if } 2 \leq i \leq d_1 + 1, \\
    d_i     & \text{if } d_1 + 2 \leq i \leq n.
\end{cases}
\end{equation}
We assemble a graph $G$ from $G_1$ if combined a new vertex $v_1$, and the $e_1$ edges $v_1 v_i$ for $2 \leq i \leq d_1 + 1$. by considering $\deg_G v_i = d_i$ for $1 \leq i \leq n$, if $d_1\geq 1$, graphical if and only if the sequence $\mathscr{B}=(d_2-1,d_3-1,\dots,d_{e_1+1}-1, d_{e_1+2},\dots,d_n)$ is graphical.

Ghalavand, A. et al. In~\cite{GhalavandAshrafiR2023} proved  $\operatorname{irr}_T(G) \leq \frac{n^2}{4} \operatorname{irr}(G)$ when the bound is sharp for infinitely many graphs.  The first and the second Zagreb index, $M_1(G)$ and $M_2(G)$ are defined in~\cite{GutmanTrinajstic1972, GutmanTrinajsticWilcox1975,Albertson1997} as: 
\[
M_1(G)=\sum_{i=1}^{n}d_i^2, \quad \text{and} \quad M_2(G)=\sum_{uv\in E(G)} d_u(G)d_v(G).
\]
Actually an alternative expressions introduced by M. Matej\'i et al. in~\cite{Albertson1997, NikolicKovacevicMilicevicTrinajstic2003} for the first Zagreb index  as $M_1(G)=\sum_{u\sim v}(d_u+d_v)$.  The recently introduced $\sigma(G)$ irregularity index is a simple diversification of the previously established Albertson irregularity index, in~\cite{GutmanToganYurttas2016,AbdoDimitrovGutman} defined as: 
\[
\sigma(G)=\sum_{uv\in E(G)}\left( d_u(G)-d_v(G) \right)^2.
\]
Denote by $\sigma_{\max}$ the maximum of Sigma index and $\sigma_{\min}$ the minimum of Sigma index.

Our goal of this paper, conduct a study on the bounds related to the selected topological indices, specifically Albertson index and the Sigma index, particularly when examining the maximum and minimum values of each index and linking these values with the first Zagreb index, the second Zagreb index and the forgotten index.

This paper is organized as follows. In Section~\ref{sec1}, we observe the important concepts to our work including literature view of most related papers, in Section~\ref{sec2} we have provided a preface through some of the important theories and properties we have utilised in understanding the work, in Section~\ref{sec3} we introduced the main result throughout the subsection~\ref{subsec3} for some bounds of Albertson index and the subsection~\ref{sub2sec3} for some bounds of Sigma index. 
 \section{Preliminaries}\label{sec2}
In this section, we explore important concepts and recent advancements in graph irregularity indices, with a particular emphasis on trees.  Let $\mathcal{S}$ be a class of graphs, then we have $\operatorname{irr}_{\max},\operatorname{irr}_{\min}$ Albertson index of a graph $G$, where: \begin{gather*} \operatorname{irr}_{\max}(\mathcal{S}) =\max \{\operatorname{irr}(G)\mid G\in \mathcal{S}\}, \\
\operatorname{irr}_{\min}(\mathcal{S}) =\min \{\operatorname{irr}(G)\mid G\in \mathcal{S}\}.
\end{gather*}
In~\cite{GhalavandAshrafiR2023} mention to total irregularity when $\sum_{i\geq 1}^{} n_i=|V(G)|$ where $|V(G)|$ is the number of vertices of degree $i$ in graph $G$ it is know as $\operatorname{irr}_T(G)=\sum_{1 \leq i<j \leq n-1} n_i n_j(j-i)$. So that in \cite{GhalavandAshrafiR2023} proved  $\operatorname{irr}_T(G) \leq \frac{n^2}{4} \operatorname{irr}(G)$. In this case, if $G$ is a tree, then $\operatorname{irr}_T(G) \leq(n-2) \operatorname{irr}(G)$. Define the ``general Albertson irregularity index'' of a graph in \cite{LinaZhoubMiaob2024} by Lina, Z., et al. such that for any number $p>0$ then: 
\[
\operatorname{irr}_p(G) = \left(\sum\limits_{uv\in E(G)}|d(u)-d(v)|^p\right)^{\frac{1}{p}}.
\]
This relationship had employed in Lemma~\ref{lemmastep2Albertsonn1} and \ref{lemmastep2Sigman1}. 
\begin{theorem}~\cite{LinaZhoubMiaob2024}
Let $G$ be a connected graph. When $\frac{1}{p} + \frac{1}{q} = 1$, then
\[
\irr_p(G) \irr_q(G) \geq F - 2M_2,
\]
with equality if and only if $p = 2$, or $G$ is a regular graph.
\end{theorem}
According to~\cite{LinaZhoubMiaob2024, ClarkSzekelyEntringer1992}, if $p \geq 1$  where $0\leq d\leq n-1$. Then, for a degree sequence $\mathscr{D}$ satisfy 
\begin{equation}~\label{eq1paper}
    \left(\sum_{i=1}^n d_i^p\right)^{\frac{1}{p}} \leq(n-1)^{1-\frac{1}{p}} \sum_{i=1}^n d_i^{\frac{1}{p}}
\end{equation}
\begin{theorem}~\cite{AliDimitrovEtAl2025}
Let $T$ be a tree of order $n\geq 4$, the total Albertson index among $T$ satisfy:
\[
\irr_t(G)=\begin{cases}
   \irr_{t_{\max}}(G) =(n-1)(n-2), \\
   \irr_{t_{\min}}(G)=2(n-2).
\end{cases}
\]
\end{theorem}
\begin{proposition}[Upper Bound Sigma index~\cite{AbdoDimitrovGutman}]
 For any graph connected $G$ with $n$ vertices with minimal degree $\delta$, maximal degree $\Delta$, and maximal $\sigma$ irregularity. Then, the upper bound of $\sigma$ as:
\begin{equation}~\label{eq5BoundsonGraph}
\sigma(G)>\frac{\delta}{\Delta+1}(\Delta-\delta)^3 n, \quad \sigma(G)>\frac{1}{\Delta+1}(\Delta-1)^3 n.
\end{equation}
 \end{proposition}
Yang J, Deng H., M.,~\cite{YangDeng2023} provide $1\leq p \leq n-3$. Then: $\Delta(G)=n-1$.  For Caterpillar tree with path vertices with degrees $d_1,d_2,\dots , d_n$, we have:
 \[\irr(G)=\left( {{d_n} - 1} \right)^2 + \left( {d_1 - 1} \right)^2 + \sum\limits_{i = 2}^{n - 1} {\left( {{d_i} - 1} \right)\left( {{d_i} - 2} \right)} +\sum_{i=1}^{n-1}|d_i-d_{i+1}|.\]
\begin{lemma}~\cite{Andriantiana2013Wagner}
Let $\mathscr{D}_1=(x_1,\dots,x_n)$,$\mathscr{D}_2=(y_1,\dots,y_n)$,$\mathscr{B}_1=(a_1,\dots,a_n)$ and $\mathscr{B}_2=(b_1,\dots,b_n)$ be a  non-increasing degrees sequences, if $\mathscr{B}_1\preccurlyeq  \mathscr{D}_1$ and $\mathscr{B}_2\preccurlyeq  \mathscr{D}_2$, then 
\begin{equation}~\label{Majorization2}
(a_1b_1,\dots,a_nb_n)\preccurlyeq (x_1y_1,\dots,x_ny_n).
\end{equation}
\end{lemma}

\begin{lemma}\cite{NasiriFathTabarGutman2013}~\label{lem1}
	The star $S_n$  is the only tree of order n that has a great deal of irregularity, satisfying:
	\[\operatorname{irr}\left( {{S_n}} \right) = \left( {n - 2} \right)\left( {n - 1} \right)\].
\end{lemma} 
\begin{lemma}~\label{lem2}
Let be $\mathcal{T}$ a classes of trees, then $\sigma_{max}, \sigma_{min}$ in trees with $n$ vertices by:
 \[\sigma(\mathcal{T})  = \left\{ \begin{array}{l}
 	{\sigma _{\max }(\mathcal{T})} = \left( {n - 1} \right)\left( {n - 2} \right){\rm{        }};n \ge 3\\
 	{\sigma _{\min }(\mathcal{T})} = 0{\rm{                             ; n = 2}}
 \end{array} \right.\].
 \end{lemma}
 According to a tree with $n$ vertices has $n-1$ edges, then we have $\sigma_{\max}(T)=n-2$. Thus, the modified total Sigma as $\sigma_t(T)=\frac{1}{2}\sum_{(u,v)\subseteq V(T)}\left(d_u-d_v\right)^2$

\begin{theorem}~\label{mainalb2}
Let  $T$ be a tree of order $n$, a degree sequence is $\mathscr{D}=(d_1,\dots,d_n)$ where $d_n\geqslant \dots \geqslant d_1$, then Albertson index of tree $T$ is: 
\[
\irr(T)=d_1^2+d_n^2+\sum_{i=2}^{n-1} d_i^2+\sum_{i=2}^{n-1} d_i+d_n - d_1-2n-2.
 \]
\end{theorem}

\section{Main Result}\label{sec3}
In this section, we presented some bounds of Albertson index via subsection~\ref{subsec3} by starting with Hypothesize~\ref{hyfibn1} of Fibonacci number of Albertson index and Hypothesize~\ref{hyfibn2} of Fibonacci number of Sigma index in subsection~\ref{sub2sec3}.
\medskip
\subsection{Some Bounds of Albertson Index}~\label{subsec3}
Through this subsection, we will present a study of some bounds according to Albertson index and will specify certain bounds between $\irr_{\max}$ and $\irr_{\min}$.  In Proposition~\ref{boundalbertsonn01}, we provide the lower bound of $\irr_{\min}$ and the upper bound of $\irr_{\max}$.
\medskip
\begin{hypothesize}~\label{hyfibn1}
Let $T$ be a tree of order $n$, let $F_n$ be a Fibonacci number with $n>2$. Let $\mathscr{D}=(F_3,\dots,F_n)$ be a degree sequence of Fibonacci number. Then, Albertson index of $T$ among $\mathscr{D}$ is 
\[
\irr(T)=\sum_{i=3}^{n-1}F_i+\sum_{i=5}^{n-1}(F_i-2)\lvert F_i-1\rvert+\lvert F_4-1\rvert+(F_n-1)\lvert F_n-1\rvert+2.
\]
\end{hypothesize}
\begin{proof}
Let $\mathscr{D}=(F_3,\dots,F_n)$ be a degree sequence of Fibonacci number, then we have
\begin{equation}~\label{eqfibn1}
\lvert F_3-F_4\rvert+\dots+\lvert F_{n-1}-F_n\rvert=\sum_{i=3}^{n-1}F_i+2.  
\end{equation}
According to Albertson index for every vertex, by considering the constant term is $\lvert F_4-1\rvert+(F_n-1)\lvert F_n-1\rvert$, we noticed that:
\begin{equation}\label{eqfibn02}
    (F_5-2)\lvert F_5-1\rvert+(F_6-2)\lvert F_6-1\rvert+\dots+(F_{n-1}-2)\lvert F_{n-1}-1\rvert=\sum_{i=5}^{n-1}(F_i-2)\lvert F_i-1\rvert.
\end{equation}
Now, by applying~\eqref{eqfibn1},\eqref{eqfibn02} to Albertson index, we obtain 
\begin{equation}~\label{eqfibn3}
\irr(T)=\sum_{i=3}^{n-1}F_i+\sum_{i=5}^{n-1}(F_i-2)\lvert F_i-1\rvert+\lvert F_4-1\rvert+(F_n-1)\lvert F_n-1\rvert+2.
\end{equation}
As desire.
\end{proof}
\begin{example}
Let $T$ be a tree of order $n$. Let $\mathscr{D}=(F_3,\dots,F_{10})$ be a degree sequence of Fibonacci number, where $F_1=1, F_2=2$ and $F_3=3$ with $F_n=F_{n-1}+F_{n-2}$, then in Figure~\ref{figdegseqfibnumn1} we observed that, then Albertson index according to~\eqref{eqfibn1} is:
\begin{equation}~\label{eqexfibn1}
F_3+\dots+F_{11}=1\sum_{i=3}^{10}F_i+2=143    
\end{equation}
Then, from~\eqref{eqfibn3},\eqref{eqexfibn1} we present Albertson index as
\[
\irr(T)=143+4430+7746=12319.
\]
\begin{figure}[H]
    \centering
    \begin{tikzpicture}[scale=1.1]
\draw   (7,5)-- (7,6);
\draw   (7,6)-- (8,7);
\draw   (8,7)-- (9,8);
\draw   (9,8)-- (10,9);
\draw   (10,9)-- (11,10);
\draw   (11,10)-- (12,9);
\draw   (12,9)-- (13,8);
\draw   (13,8)-- (14,7);
\draw   (14,7)-- (15,6);
\draw   (8,7)-- (8,6);
\draw   (9,8)-- (8.537909109638452,7.0048314346548555);
\draw   (9,8)-- (9,7);
\draw   (9,8)-- (9.564119435380462,7.023835329576004);
\draw   (10,9)-- (9.48810385569587,8.031041760396866);
\draw   (10,9)-- (10.381286916989843,7.974030075633421);
\draw (9.640135015065058,8.2) node[anchor=north west] {$\dots$};
\draw   (12,9)-- (11,8);
\draw   (12,9)-- (12,8);
\draw (11.179450503678076,8.2) node[anchor=north west] {$\dots$};
\draw   (13,8)-- (12,7);
\draw   (13,8)-- (13,7);
\draw (12.224664724341235,7.2) node[anchor=north west] {$\dots$};
\draw   (14,7)-- (13,6);
\draw   (14,7)-- (14,6);
\draw (13.250875050083247,6.1) node[anchor=north west] {$\dots$};
\draw   (15,6)-- (14,5);
\draw   (15,6)-- (15,5);
\draw (14.220073691061812,5.2) node[anchor=north west] {$\dots$};
\draw (6.219433929258355,6.5867457463895915) node[anchor=north west] {$F_3$};
\draw (7.340663729606107,7.726979441658493) node[anchor=north west] {$F_4$};
\draw (8.34787016042697,8.715181977558206) node[anchor=north west] {$F_5$};
\draw (9.222049326799794,9.627368933773328) node[anchor=north west] {$F_6$};
\draw (10.495310286516736,10.729594839199931) node[anchor=north west] {$F_7$};
\draw (12.034625775129753,9.874419567748257) node[anchor=north west] {$F_8$};
\draw (13.155855575477505,8.867213136927393) node[anchor=north west] {$F_9$};
\draw (14.049038636771478,7.917018390869976) node[anchor=north west] {$F_{10}$};
\draw (15.227280121882675,6.757780800679927) node[anchor=north west] {$F_{11}$};
\draw (10.685349235728218,9.532349459167586) node[anchor=north west] {$\dots$};
\begin{scriptsize}
\draw [fill=black] (7,5) circle (1.5pt);
\draw [fill=black] (7,6) circle (1.5pt);
\draw [fill=black] (8,7) circle (1.5pt);
\draw [fill=black] (9,8) circle (1.5pt);
\draw [fill=black] (10,9) circle (1.5pt);
\draw [fill=black] (11,10) circle (1.5pt);
\draw [fill=black] (12,9) circle (1.5pt);
\draw [fill=black] (13,8) circle (1.5pt);
\draw [fill=black] (14,7) circle (1.5pt);
\draw [fill=black] (15,6) circle (1.5pt);
\draw [fill=black] (8,6) circle (1.5pt);
\draw [fill=black] (8.537909109638452,7.0048314346548555) circle (1.5pt);
\draw [fill=black] (9,7) circle (1.5pt);
\draw [fill=black] (9.564119435380462,7.023835329576004) circle (1.5pt);
\draw [fill=black] (9.48810385569587,8.031041760396866) circle (1.5pt);
\draw [fill=black] (10.381286916989843,7.974030075633421) circle (1.5pt);
\draw [fill=black] (11,8) circle (1.5pt);
\draw [fill=black] (12,8) circle (1.5pt);
\draw [fill=black] (12,7) circle (1.5pt);
\draw [fill=black] (13,7) circle (1.5pt);
\draw [fill=black] (13,6) circle (1.5pt);
\draw [fill=black] (14,6) circle (1.5pt);
\draw [fill=black] (14,5) circle (1.5pt);
\draw [fill=black] (15,5) circle (1.5pt);
\end{scriptsize}
\end{tikzpicture}
    \caption{Degree sequence of Fibonacci number}
    \label{figdegseqfibnumn1}
\end{figure}
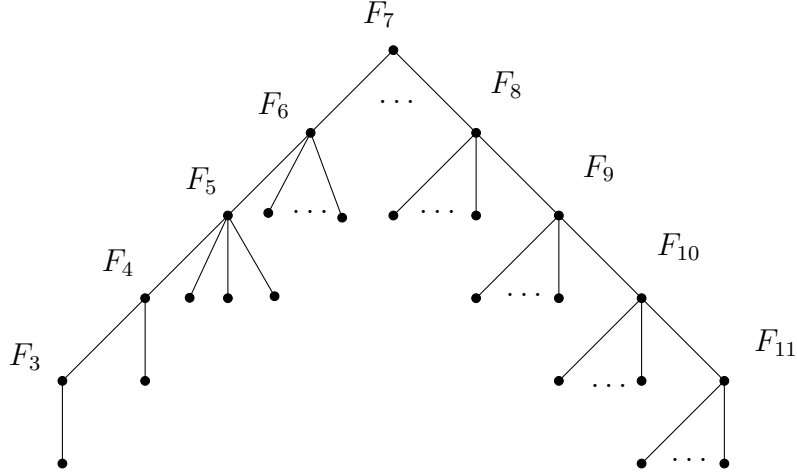
\end{example}
\begin{proposition}~\label{boundalbertsonn01}
Let $T$ be a tree of order $n$, let $\mathscr{D}=(d_1,\dots,d_n)$ be an increasing degree sequence, the maximum Albertson index $\irr_{\max}$, the minmum Albertson index $\irr_{\min}$, let $\Delta$ the maximum degree of vertices ($\delta$ the minmum degree of vertices), then 
\begin{equation}~\label{f1boundIrrn1}
\irr_{\min}<2n\delta\frac{2(nm)^3+2m\Delta^2}{8n^4\Delta+8m^3\delta\Delta+\Delta^2(\Delta-1)}\leq \irr_{\max}.
\end{equation}
\end{proposition}
\begin{proof}
Let $T$ be a tree of order $n$, let $\mathscr{D}=(d_1,\dots,d_n)$ be a degree sequence where $d_n\geqslant \dots \geqslant d_1$, then we know $\irr_{\min}<\irr_{\max}$. Thus, we need to prove~\eqref{f1boundIrrn1}, since a tree has $m=n-1$ edges. Then, 
\begin{equation}~\label{f1boundIrrn2}
\irr_{\min}<\frac{2(nm)^3+2m\Delta^2}{8n^3\Delta+8m^3\Delta+\Delta^2(\Delta-1)}\leq \irr_{\max}.
\end{equation}
Then, a degree sequence must satisfy sum of degrees $2m$,  we need to interpret bound~\eqref{f1boundIrrn2}, by considering the term $2n\delta$. Thus, from~\eqref{f1boundIrrn2} we have
\begin{equation}~\label{f1boundIrrn3}
\irr_{\min}<\frac{\binom{m}{2\Delta}+2(nm)^3+2m\Delta^2}{8n^3\Delta+8m^3\Delta+\Delta^2(\Delta-1)+\sqrt{m(\Delta-1)}}\leq \irr_{\max}.    
\end{equation}
Then, according to~\eqref{f1boundIrrn2},\eqref{f1boundIrrn3} we have: 
\begin{equation}~\label{f1boundIrrn4}
\irr_{\min}<\frac{\Delta^2(\Delta-1)+nm^2}{6(\Delta-1)+m+9\delta}\leq \irr_{\max}.    
\end{equation}
Therefore, according to required bound~\eqref{f1boundIrrn1}, clearly from bounds~\eqref{f1boundIrrn2}--\eqref{f1boundIrrn4}, we observe that for the maximum degree $\Delta$ and for vertices $n$ and edges $m$, when $\irr_{\min}<\irr_{\max}$, then~\eqref{f1boundIrrn1} is valid. As desire.
\end{proof}
Actually, from Proposition~\ref{boundalbertsonn01}, we presented Proposition~\ref{boundalbertsonn02} for the upper bound and the lower bound of Albertson index. Additionally, in Proposition~\ref{lowerboundAlbertsonstep2} for the lower bound.
\begin{proposition}~\label{boundalbertsonn02}
Let $T$ be a tree of order $n$, let $\mathscr{D}=(d_1,\dots,d_n)$ be an increasing degree sequence, the maximum Albertson index $\irr_{\max}$, the minmum Albertson index $\irr_{\min}$. Then, 
\begin{itemize}
    \item The upper bound of the minmum Albertson index is: 
    \begin{equation}~\label{eq1boundalbertsonn02}
\irr_{\min}\geqslant \frac{(\Delta-2)^3}{n\Delta-\delta}.
\end{equation}
\item The lower bound of the minmum Albertson index is: 
    \begin{equation}~\label{eq2boundalbertsonn02}
\irr_{\min}\leqslant \frac{\delta}{\Delta+1}n\Delta^2+\frac{\Delta^2(\Delta-\delta)}{6\delta(\Delta-1)}.
\end{equation}
\end{itemize}
\end{proposition}
\begin{proof}
Let $\mathscr{D} = (d_1, \dots, d_n)$ be an increasing degree sequence, then we need to prove the upper bound~\eqref{eq1boundalbertsonn02} and the lower bound~\eqref{eq2boundalbertsonn02}. Thus, possibly combining the terms if they share a common structure, for the upper bound~\eqref{eq1boundalbertsonn02} noticed that $\irr_{\min}> \delta(T)$ by considering $\Delta-\delta>1$ when $\delta(T)\leq (\Delta-2)^3/\Delta$. This term is straightforward, then we obtained the lower bound of $\delta(T)$ with upper bound of the minmum Albertson index as
\begin{equation}~\label{eq3boundalbertsonn02}
   \irr_{\min}\leq \frac{1}{\Delta+1}(\Delta-1)^3n.
\end{equation}
Therefore, the bound~\eqref{eq3boundalbertsonn02} holds the bound~\eqref{eq1boundalbertsonn02} by considering $(\Delta-2)^3-(n\Delta-\delta)<3$, then the bound~\eqref{eq1boundalbertsonn02} is valid. Now, we need to confirm the lower bound~\eqref{eq2boundalbertsonn02} of the minmum Albertson index, thus, according to Proposition~\ref{boundalbertsonn01} it is clear to define the upper bound of the minmum Albertson index as 
\begin{equation}~\label{eq4boundalbertsonn02}
    \irr_{\min} \geq \frac{\Delta^2(\Delta-\delta)}{6\delta(\Delta-1)}.
\end{equation}
Thus, we obtained on the lower bound as
\begin{equation}~\label{eq5boundalbertsonn02}
\irr_{\min} \leq \frac{n\Delta^2}{\Delta+1}.  
\end{equation}
Then, we have 
\begin{equation}~\label{eq6boundalbertsonn02}
 \irr_{\min} \leq \frac{\Delta^2 \left[ \Delta^2 + \left(6 \delta^2 n + 1 - \delta \right) \Delta - 6 \delta^2 n - \delta \right]}{6 \delta \left(\Delta^2 - 1 \right)},
\end{equation}
Further
\[
\irr_{\min} \leq \frac{\Delta^2 \left( \Delta^2 + \left(6 \delta^2 n + 1 - \delta \right) \Delta - 6 \delta^2 n - \delta \right)}{6 \delta \left(\Delta^2 - 1 \right)}.
\]
From~\eqref{eq4boundalbertsonn02}--\eqref{eq6boundalbertsonn02} we observe the bound~\eqref{eq2boundalbertsonn02} is valid. As desire.
\end{proof}
\begin{proposition}~\label{lowerboundAlbertsonstep2}
Let $T=(V,E)$ be a tree, where $n=|V|$ vertices and $m=|E|$ with maximum degree $\Delta$ and minmum degree $\delta$, then the upper bound of the minmum Albertson index $\irr_{\min}$ is 
\begin{equation}~\label{eq1lowerboundAlbertsonstep2}
\irr_{\min}\geq \ceil{\log_2(n+1)-1}+2\Delta(T)-1.
\end{equation}
\end{proposition}
\begin{proof}
Let $T=(V,E)$ be a tree, where $n=|V|$ vertices, assume $n>2$, let the maximum degree of any vertex $\ell$, where $v_{\ell}\leq 2^{\ell}$, then the sum of degrees is equals $2m$, to prove~\eqref{eq1lowerboundAlbertsonstep2} then we have the bound 
\begin{equation}~\label{eq2lowerboundAlbertsonstep2}
\sum_{i=1}^{n}\deg_T(v_i)\leq 2^{\ell+1}-1.
\end{equation}
Actually, according to Proposition~\ref{boundalbertsonn01} it is clear to define the upper bound~\eqref{eq4boundalbertsonn02} of the minmum Albertson index as 
\begin{equation}~\label{eq3lowerboundAlbertsonstep2}
   2^{\ell+1}-1 \leq \frac{\Delta^2(\Delta-\delta)}{6\delta(\Delta-1)}.
\end{equation}
In a tree $T$ with $n$ vertices and maximum degree $\Delta(T)$, the upper bound of the minimum Albertson index $\irr_{\min}$ may be bounded by $\lceil \log_2(n+1) - 1 \rceil + 2\Delta(T) - 1$. Consequently, from~\eqref{eq2lowerboundAlbertsonstep2} and \eqref{eq3lowerboundAlbertsonstep2}  when $n>2$, the bound of $\ell$ provided as
\begin{equation}~\label{eq4lowerboundAlbertsonstep2}
\ell \leq \log_2(n+1)-1+2\Delta(T)-1.
\end{equation}
As desire.
\end{proof}

Lemma~\ref{lemmaboundAlbertsonn1} considers a pair of trees, $T_1$ and $T_2 $, with distinct degree sequences: $T_1$ possesses a degree sequence dominated by vertices of degree 2, resembling a path graph, while $T_2$ is characterized by a degree sequence comprising prime numbers. In Lemma~\ref{lemmaboundAlbertsonn2} focuses on a single tree with general degree properties, introducing a parameterized lower bound that accounts for the maximum and minimum degrees, modulated by a scaling factor $\alpha$. In both Lemma's, we provide a bounds among Albertson index.

\begin{lemma}~\label{lemmaboundAlbertsonn1}
Let $T_1,T_2$ be a tow trees, where $T_1$ have $\mathscr{D}=(d_1,\dots, d_i)$ an increasing degree sequence occur of $2$. Tree $T_2$ have $\mathscr{B}=(b_1,\dots, b_i)$ an increasing degree sequence of prime number, consider $n_1=|V(T_1)|,n_2=|V(T_2)|$ and $m_1=|E(T_1)|,m_2=|E(T_2)|$. Then, $n_2=n_1-1$, $m_2=m_1-1$ and $\Delta(T_2)=\Delta(T_1)+1$. Then the bound of Albertson index given by: 
\begin{equation}~\label{eq1lemmaboundAlbertsonn1}
\irr_{\min}\leq  5\cdot \frac{n_1\Delta_2^3+n_2\Delta_1^4+m_1\Delta_2^2}{\Delta_1(\Delta_1+\Delta_2)^2} \leq \irr_{\max}.
\end{equation}
\end{lemma} 
\begin{proof}
Let $T_1,T_2$ be a tow trees. It is not necessary for both trees to be a star or a path; rather, they are general trees of the same type. Denote by $\irr_{\min}=\irr_{\min}(T_1)+\irr_{\min}(T_2)$ and $\irr_{\max}=\irr_{\max}(T_1)+\irr_{\max}(T_2)$. The bound~\eqref{eq1lemmaboundAlbertsonn1} holds the bound of minmum Albertson index given by: 
\begin{equation}~\label{eq2lemmaboundAlbertsonn1}
\irr_{\min}\geq \frac{n_1\Delta_2^3+n_2\Delta_1^4+m_1\Delta_2^2}{\Delta_1(\Delta_1+\Delta_2)^2}.
\end{equation}
According to Proposition~\ref{boundalbertsonn01} from~\eqref{f1boundIrrn1} we find the bound is equivalent to its counterpart in~\eqref{eq1lemmaboundAlbertsonn1} by considering we dealing with $T_1,T_2$ given by term $\irr_{\min}=\irr_{\min}(T_1)+\irr_{\min}(T_2)$ and $\irr_{\max}=\irr_{\max}(T_1)+\irr_{\max}(T_2)$. We can ensure the maximum degree by using a structure a tree with bounded degree as: 
\[
\irr_{\min}<2n\delta\frac{2(nm)^3+2m\Delta^2}{8n^4\Delta+8m^3\delta\Delta+\Delta^2(\Delta-1)}\leq \irr_{\max}.
\]
Thus, if $T_1\cong T_2$, as we know every node in a binary tree can have zero, one, or two child nodes, referred to as the left node and the right node.  Thus, both bounds, the bounds~\eqref{eq3lemmaboundAlbertsonn1} and \eqref{eq4lemmaboundAlbertsonn1} are equivalent to its counterpart in~\eqref{eq2lemmaboundAlbertsonn1}  as
\begin{equation}~\label{eq3lemmaboundAlbertsonn1}
\irr_{\min}\leq \frac{n_1\Delta_2^3+n_2\Delta_1^4+m_1\Delta_2^2}{\Delta_1(\Delta_1-1)^2} \leq \irr_{\max}.
\end{equation}
On the other hand, even if both $T_1$ and $T_2$ are not isomorphic $T_1\ncong T_2$, it is according to the maximum degree $\Delta$ in each of them by considering the difference between $\Delta_1$ and $\Delta_2$ no more 2,  then
\begin{equation}~\label{eq4lemmaboundAlbertsonn1}
\irr_{\min}\leq \frac{n_1\Delta_2^3+n_2\Delta_1^4+m_1\Delta_2^2}{\Delta_2(\Delta_2-1)^2} \leq \irr_{\max}.
\end{equation}
Therefore, from~\eqref{eq2lemmaboundAlbertsonn1} noticed that the sum of degrees $d_1,d_2$ it is similarly with the sum of degrees $b_1,b_2$. As desire.
\end{proof}
\begin{example}
Let $T_1$ and $T_2$ be trees with:
\begin{itemize}
    \item $n_1 = 38$ vertices, $m_1 = 37$ edges, and maximum degree $\Delta_1 = 12$
    \item $n_2 = 37$ vertices, $m_2 = 36$ edges, and maximum degree $\Delta_2 = 13$
\end{itemize}
Then, according to Lemma~\ref{lemmaboundAlbertsonn1} we have: $\irr_{\min}=562$ and $\irr_{\max}=612$.
\end{example}

\begin{lemma}~\label{lemmaboundAlbertsonn01}
Let $T$ be a tree with the maximum of Albertson index $\irr_{\max}$ and the minmum Albertson index $\irr_{\min}$. Let $\mathscr{D}=(d_1,\dots,d_n)$ be an increasing degree sequence, where $n=|V(T)|$ and $m=|E(T)|$, then 
\begin{equation}~\label{eq1lemmaboundAlbertsonn01}
\begin{cases}
\irr_{\max} <\Delta(T)(\irr_{\min}-n),\\
\irr_{\min} <\Delta(T)(\irr_{\max}-n).
\end{cases}  
\end{equation}
\end{lemma}
\begin{proof}
Let $T$ be a tree of order at least 3, the fact $m<3(\irr(T)-n)$ is valid for $m$ and $n<3(\irr(T)-m)$ is valid for $n$.  Assume $\mathscr{D}=(d_1,\dots,d_n)$ be an increasing degree sequence, then $m\leq \Delta(T)(n/2)$, by restricting the maximum
degree $\Delta(T)$ where $\irr(T)\leq n\Delta^2-m$. Then, the lower bound of the minmum Albertson index is 
\begin{equation}~\label{eq2lemmaboundAlbertsonn01}
m\irr_{\max}\leq  \Delta(T)(n/2) \irr_{\min} \quad \text{and} \quad m\irr_{\min}\leq  \Delta(T)(n/2) \irr_{\max}.
\end{equation}
Then, according to Theorem~\ref{mainalb2} for an increasing degree sequence $\mathscr{D}$ consecutive terms in a degree sequence from~\eqref{eq2lemmaboundAlbertsonn01} by using the constant term $(\Delta+5)n/2+m-7$. This term scales the maximum vertex degree by a factor proportional to the graph size, we noticed that: 
\begin{equation}~\label{eq3lemmaboundAlbertsonn01}
 \irr(T)\leq \sum_{i=1}^{n}\lvert d_i-d_{i+1}\rvert +(\Delta-5)\frac{n}{2}+\sum_{i=2}^{n-1}\lvert d_i+2\rvert \lvert d_i-1\rvert+m-7.
\end{equation}
From the bound~\eqref{eq2lemmaboundAlbertsonn01} and for the minimum Albertson index according to~\eqref{eq3lemmaboundAlbertsonn01} we have
\begin{equation}~\label{eq4lemmaboundAlbertsonn01}
 \irr(T)=\sum_{i=1}^{n}\lvert d_i-d_{i+1}\rvert +\Delta(T)\frac{n}{5}+\sum_{i=2}^{n-1}\lvert d_i+2\rvert \lvert d_i-1\rvert+m+2.
\end{equation}
Therefore, $\irr_{\max}-\irr_{\min}\leq \floor{n/2}$, then $\irr_{\max} \leq \irr_{\min} \floor{n/2}$ and $\irr_{\min} \leq \irr_{\max} \floor{n/2}$. Then, $\irr_{\min} \leq \Delta(\irr_{\max}- \floor{n/2})$ and $\irr_{\max} \leq \Delta(\irr_{\min}- \floor{n/2})$. Thus, the bound~\eqref{eq2lemmaboundAlbertsonn01} is valid. As desire.
\end{proof}
\begin{lemma}~\label{lemmaboundAlbertsonn2}
Let $T$ be a tree with $n=|V(T)|,m=|E(T)$  with the maximum degree $\Delta$ and the minimum degree $\delta\geqslant 1$ and $0\leq \alpha \leq 1$. Then, the lower bound of Albertson index is 
 \[\irr(T) \leq \floor*{\frac{3n^2-10n}{2}}2^{\alpha}+\Delta^2-n\delta.\]
\end{lemma}
\begin{proof}
Let $T$ be a tree with $\Delta\geqslant 3$, then we find that $\irr(T)\leq 3(\Delta^2-n)$ by consideration $\Delta >\delta$. Then, the lower bound given by: 
\begin{equation}~\label{eq1lemmaboundAlbertsonn2}
\irr(T)\leq \floor*{\dfrac{3n^2}{\Delta}}.
\end{equation}
Thus, if $T$ be a star tree $\mathcal{S}_n$, then we know $\irr(\mathcal{S}_n)=(n-1)(n-2)$. In this case, by considering $0\leq \alpha \leq 1$ where each edge of the star is adjacent to every other edge at the central vertex. Then, $\alpha\geq \log_2(((n-1)(n-2) - \Delta^2 + n)/(\floor{(3n^2 - 10n)/2}))$, then to evaluate this inequality~\eqref{eq1lemmaboundAlbertsonn2} the inequality holds across the range of $\alpha$ as
\begin{equation}~\label{eq2lemmaboundAlbertsonn2}
\irr(\mathcal{S}_n)  \leq \floor*{\frac{3n^2-10n}{2}}2^{\alpha}+\Delta^2-n.
\end{equation}
Now, if $T$ be a path $\mathcal{P}_n$. Each end vertex is connected to a vertex of degree 2, contributing all pairs of adjacent vertices $\irr(\mathcal{P}_n)=2$, we need to  provide it for all vertices, for that from the bound~\eqref{eq1lemmaboundAlbertsonn2} holds the bound 
\begin{equation}~\label{eq3lemmaboundAlbertsonn2}
\irr(\mathcal{P}_n)\leq \floor*{\dfrac{3n^2-10n+\Delta^2}{\Delta}}.
\end{equation}
Thus, let $\alpha=\dfrac{n}{\Delta^2-1}$, then for the term $2^{n}$ is growth of the bound~\eqref{eq1lemmaboundAlbertsonn2} and \eqref{eq3lemmaboundAlbertsonn2} for this case, the term $1\leq 2^{n/(\Delta^2-1)}\leq 2$ is growth the bounds of Albertson index. Since $\Delta$ is typically a positive and $\Delta>\delta$, we need to prove the bound
\begin{equation}~\label{eq6lemmaboundAlbertsonn2}
\irr(\mathcal{P}_n)  \leq \floor*{\frac{3n^2-10n}{2}}2^{\alpha}+\Delta^2-n.
\end{equation}
In this case, we need to determine if 
\begin{equation}~\label{eq4lemmaboundAlbertsonn2}
\floor*{\dfrac{3n^2}{\Delta}}-\floor{\dfrac{3n^2-10n+\Delta^2}{\Delta}}
\end{equation}
satisfy the required related relationship, let be organized~\eqref{eq4lemmaboundAlbertsonn2} by define $a =3n^2/\Delta,  b = 10n/\Delta $ and $c = \Delta$. Then, we have $a = \left\lfloor a \right\rfloor + \{a\}$ where $0 \leq \{a\} < 1$, 
\begin{align*}
\floor{a} - \floor{a - b + c}&=\floor{a}-\floor{ \floor{a}+ \{a\} - b + c}\\
&=\floor{a}-\floor{a}+\floor{\{a\} - b + c}\\
&=-\floor{\{a\} - b + c}
\end{align*}
Therefore, the bound~\eqref{eq4lemmaboundAlbertsonn2} holds
\begin{equation}~\label{eq5lemmaboundAlbertsonn2}
   \floor*{\left\{ \frac{3n^2}{\Delta} \right\} - \frac{10n}{\Delta} + \Delta}>0.
\end{equation}
Thus, for additional terms to the bound~\eqref{eq6lemmaboundAlbertsonn2} is valid, so that the term $\Delta^2-n$ satisfy
\begin{equation}~\label{eq7lemmaboundAlbertsonn2}
  \irr(T)\leq \floor*{\dfrac{3n^2}{\Delta}}+  \Delta^2-n.
\end{equation}
As desired.
\end{proof}
The second Zagreb index $M_2(T)$ and the forgotten index $F(T)$ had defined in~\cite{GutmanToganYurttas2016,FurtulaGutman2015} as 
\begin{equation}~\label{defsecondZagreb}
M_2(T)=\sum_{uv\in E(T)}d_ud_v, \quad F(T)=\sum_{u \in V(G)} (d_u)^3 = \sum_{uv \in E(G)} [(d_u)^2 + (d_v)^2].
\end{equation}
In Proposition~\ref{boundstep2Albertsonn1} we present an upper bound~\eqref{eq1boundstep2Albertsonn1} of Albertson index with the second Zagreb index $M_2(T)$ and the forgotten index $F(T)$.  

\begin{proposition}~\label{boundstep2Albertsonn1}
Let $T$ be a tree, $M_2(T)$ be the second Zagreb index and $F(T)$ the forgotten index. Then, the upper bound of Albertson index is: 
\begin{equation}~\label{eq1boundstep2Albertsonn1}
\irr(T)\geq  \sqrt{\frac{F(T)+2M_2(T)-n\Delta}{\Delta(\Delta-1)}}.
\end{equation}
\end{proposition}
\begin{proof}
Let $M_2(T)$ be the second Zagreb index and $F(T)$ the forgotten index given in~\eqref{defsecondZagreb} with maximum degree $\Delta$. Then, the general Albertson irregularity index of graph $G$ given in~\cite{LinaZhoubMiaob2024} as 
\begin{equation}~\label{eq2boundstep2Albertsonn1}
\irr_p(G)=\left(\sum_{uv\in E(G)}\lvert d_u-d_v\rvert^p  \right)^{\frac{1}{p}}.
\end{equation}
When $1<p<2$, then~\eqref{eq2boundstep2Albertsonn1} holds  $\irr(T)\leq m \irr_p(T)$ by considering $\lvert d_u-d_v\rvert \neq 0$, is indeed the leading term, when $p\geq 2$, $\irr_p(G)\leq \sqrt{F(T)-2M_2(T)}$. If $p \cdot \irr(T) \ll m$, by using a first-order Taylor we have: 
\[
\left(m + p \cdot \irr(T)\right)^{1/p} \approx m^{1/p} + \frac{\irr(T)}{m^{(p-1)/p}}, \quad \left( \frac{\irr(T)}{m^{(p-1)/p}} \right)^p = \frac{\irr(T)^p}{m^{p-1}}.
\]
Then, the bound of~\eqref{eq2boundstep2Albertsonn1} is 
\begin{equation}~\label{eq3boundstep2Albertsonn1}
\irr(T)\geq \frac{F(T)+2M_2(T)}{\Delta}.
\end{equation}
Now, we need to show the bound~\eqref{eq4boundstep2Albertsonn1} according to the bound~\eqref{eq3boundstep2Albertsonn1} as
\begin{equation}~\label{eq4boundstep2Albertsonn1}
\frac{F(T)+2M_2(T)}{\Delta}- \frac{F(T)+2M_2(T)-n\Delta}{\Delta(\Delta-1)},
\end{equation}
satisfy the term required in~\eqref{eq1boundstep2Albertsonn1}. Thus, this matches bounds~\eqref{eq3boundstep2Albertsonn1} and \eqref{eq2boundstep2Albertsonn1}, confirming the correctness
\begin{align*}
\frac{F(T)+2M_2(T)}{\Delta}- \frac{F(T)+2M_2(T)-n\Delta}{\Delta(\Delta-1)}&= \frac{(F(T) + 2M_2(T))(\Delta - 1) - (F(T) + 2M_2(T) - n\Delta)}{\Delta(\Delta - 1)}\\
&=\frac{(F(T) + 2M_2(T))(\Delta - 2) + n\Delta}{\Delta(\Delta - 1)}>0.
\end{align*}
For extended the bounds for $n$, let $\mathscr{D}=(d_1,\dots,d_n)$ be a degree sequence where $\sum_{i=1}d_i\leqslant n$. Then, by considering $\irr(T)\geq \sqrt{n}$, with the constant term $\Delta(\Delta-1)$ the relationship 
\[
\frac{F(T)+2M_2(T)-n\Delta}{\Delta(\Delta-1)}\leq n.
\]
As desire.
\end{proof}

\begin{lemma}~\label{lemmastep2Albertsonn1}
Let $T \in \mathcal{T}_{n, \Delta}$. Then,let $v_0 \in V(T)$ be a vertex with maximum degree $\Delta$. For any vertex $v_{\ell}$ in $T$, different from $v_0$, where $\deg(v_{\ell})\geqslant 3$ satisfy 
\[
2^\lambda \leq \irr(T)\leq (n-1)(n-2)^\lambda.
\]
The lower bound holds if and only if $T\cong \mathcal{P}_n$. The upper bound holds if and only if $T\cong K_{1,n-1}$.
\end{lemma}
\begin{proof}
Let $T$ be a tree of order $n\geq 3$, with maximum degree $3\leq \Delta\leq n-1$. Let $v_0 \in V(T)$ be a vertex with maximum degree $\Delta$. A vertex $v_{\ell}$ in $T$, different from $v_0$, where $\deg(v_{\ell})=\lambda\geqslant 3$ Then, when $\irr(T)\leq n^\lambda$, if $T\cong K_{1,n-1}$ for the star tree we have $\irr(\mathcal{S}_n)=(n-1)(n-2)$, we compute $\irr(\mathcal{S}_n)\leq (n-1)(n-2)^\lambda$. Then,
\begin{equation}~\label{eq1lemmastep2Albertsonn1}
\irr(T)\leq (n-1)(n-2)^\lambda.
\end{equation}
Consequently, Albertson index satisfy $\irr(T)\leq \Delta(n-1)^\lambda \leq (n-1)(n-2)^\lambda=\irr(K_{1,n-1})$.   

When $\deg(v_{\ell})=\lambda\geqslant 3$ with $3\leq \Delta\leq n-1$. Then, in~\cite{LinaZhoubMiaob2024} we find that for the general Albertson index $2^{1/p}\leq \irr_p(T)\leq (n-2)(n-1)^{1/p}$. For the lower bound when the tree is path $\mathcal{P}_n$ of order $n$ satisfy  $\irr(T)>3^\lambda>2^\lambda=\irr(\mathcal{P}_n)$.
\end{proof}

\medskip
\subsection{Some Bounds of Sigma Index}~\label{sub2sec3}
Through this subsection, we will present a study of some bounds according to Sigma index and will specify certain bounds between $\sigma_{\max}$ and $\sigma_{\min}$ which we show that in Proposition~\ref{boundSigman01} and Proposition~\ref{boundSigman02}.
\medskip
\begin{hypothesize}~\label{hyfibn2}
Let $T$ be a tree of order $n$, let $F_n$ be a Fibonacci number with $n>2$. Let $\mathscr{D}=(F_3,\dots,F_n)$ be a degree sequence of Fibonacci number. Then, Sigma index of $T$ among $\mathscr{D}$ is 
\[
\irr(T)=\sum_{i=3}^{n-1}F_i+\sum_{i=5}^{n-1}(F_i-2)\lvert F_i-1\rvert+\lvert F_4-1\rvert+(F_n-1)\lvert F_n-1\rvert+2.
\]
\end{hypothesize}
\begin{proof}
Let $\mathscr{D}=(F_3,\dots,F_n)$ be a degree sequence of Fibonacci number, according to Hypothesize~\ref{hyfibn1} for Albertson index, noticed that Sigma index for measuring between vertices as
\begin{equation}~\label{eqsigmafibn1}
(F_3-F_4)^2+\dots+(F_{n-1}-F_n)^2=\sum_{i=3}^{n-1}F_i^2+2.  
\end{equation}
According to Sigma index for every vertex, by considering the constant term is $(F_4-1)^2+(F_n-1)(F_n-1)^2$, we noticed (see Figure~\ref{figdegseqfibnumn1}) that:
\begin{equation}\label{eqSigmafibn02}
    (F_5-2)( F_5-1)^2+\dots+(F_{n-1}-2)( F_{n-1}-1)^2=\sum_{i=5}^{n-1}(F_i-2)( F_i-1)^2.
\end{equation}
Now, by applying~\eqref{eqsigmafibn1},\eqref{eqSigmafibn02} to Albertson index, we obtain 
\begin{equation}~\label{eqSigmafibn3}
\irr(T)=\sum_{i=3}^{n-1}F_i^2+\sum_{i=5}^{n-1}(F_i-2)( F_i-1)^2+( F_4-1)^2+(F_n-1)( F_n-1)^2+2.
\end{equation}
As desire.
\end{proof}
\begin{proposition}~\label{boundSigman01}
Let $T$ be a tree of order $n$, let $\mathscr{D}=(d_1,\dots,d_n)$ be a non-increasing degree sequence, the maximum Sigma index $\sigma_{\max}$, the minmum Sigma index $\sigma_{\min}$, let $\Delta$ the maximum degree of vertices ($\delta$ the minmum degree of vertices), let $M_1(T)$ be the first Zagreb index. Then,  
\begin{equation}~\label{f1boundSigman1}
\sigma_{\min} \leq n \frac{mn(\Delta^2+2\Delta)+2m\delta\Delta}{(\Delta^2-2)(\Delta-1)^2+2n\Delta\sqrt{m\Delta^2}} (n+2) \log \binom{M_1(T)}{2n} \leq \sigma_{\max}.
\end{equation}
\end{proposition}
\begin{proof}
Let $T$ be a tree of order $n$, let $\mathscr{D}=(d_1,\dots,d_n)$ be a degree sequence where $d_n\geqslant \dots \geqslant d_1$, then we know $\sigma_{\min}<\sigma_{\max}$. Thus, we need to prove~\eqref{f1boundIrrn1}, then, the term $2mn(\Delta^2+2\Delta)/(2n\Delta\sqrt{m\Delta^2})$ it represents a sufficiently small value,  declare the bound 
\begin{equation}~\label{f1boundSigman2}
\sigma_{\min}<\frac{\binom{n}{2\Delta}+n^2\delta}{\Delta-1}\leq \sigma_{\max}.
\end{equation}
Then, since a tree has $m=n-1$ edges. Additionally, a tree's degree sequence must satisfy sum of degrees $2m$,  we need to interpret bound~\eqref{f1boundSigman2} this is the complete degree sequence for all vertices, by considering the first Zagreb index as a constant term. Thus, from~\eqref{f1boundSigman2} we have
\begin{equation}~\label{f1boundSigman3}
\sigma_{\min}<\frac{\binom{m}{2\Delta}+n^2\delta+(\Delta+4)^2}{(\Delta-1)\sqrt{m\Delta}}\leq \sigma_{\max}.    
\end{equation}
Then, according to~\eqref{f1boundSigman2},\eqref{f1boundSigman3} we have: 
\begin{equation}~\label{f1boundSigman4}
\sigma_{\min}<\frac{\binom{m}{2\Delta}n^2+n^2\delta+(\Delta+4)^2}{(\Delta^2-2)\sqrt{m(\Delta-1)}}\leq \sigma_{\max}.    
\end{equation}
Therefore, according to required bound~\eqref{f1boundSigman1}, clearly $\log{n!}=n\log{n}+(1/2)\log(2\pi n)$, then the bound for maximum Sigma index given by 
\begin{equation}~\label{f1boundSigman5}
   \sigma_{\max}\geqslant (n+2) \log \binom{\sum_{i=1}^{n}d_i^2}{2n}.
\end{equation}
Thus, the bound~\eqref{f1boundSigman5} holds the bound 
\begin{equation}~\label{f1boundSigman6}
   \sigma_{\max}\geqslant (n+2) \log \binom{\sum_{i=1}^{n}d_i^2}{2n} \frac{\binom{m}{2\Delta}n^2+n^2\delta+(\Delta+4)^2}{(\Delta^2-2)\sqrt{m(\Delta-1)}}.
\end{equation}
Finally, from bounds~\eqref{f1boundSigman2}--\eqref{f1boundSigman6}, we observe that for the maximum degree $\Delta$ and for vertices $n$ and edges $m$, when $\sigma_{\min}<\sigma_{\max}$, then~\eqref{f1boundSigman1} is valid. As desire.
\end{proof}
\begin{proposition}~\label{boundSigman02}
Let $T$ be a tree of order $n$, let $\mathscr{D}=(d_1,\dots,d_n)$ be an increasing degree sequence, the maximum Sigma index $\sigma_{\max}$, the minmum Sigma index $\sigma_{\min}$. Then, 
\begin{itemize}
    \item The upper bound of the minmum Sigma index is: 
    \begin{equation}~\label{eq1boundSigman02}
\sigma_{\min}\geqslant \frac{(\Delta-2)^3}{n\Delta-\delta}.
\end{equation}
\item The lower bound of the minmum Sigma index is: 
    \begin{equation}~\label{eq2boundSigman02}
\sigma_{\min}\leqslant \frac{\delta}{\Delta+1}n\Delta^2+\frac{\Delta^2(\Delta-\delta)}{6\delta(\Delta-1)}.
\end{equation}
\end{itemize}
\end{proposition}
\begin{proof}
Given an increasing degree sequence $\mathscr{D}=(d_1,\dots,d_n)$, according to Proposition~\ref{boundalbertsonn02} we find that for the inequality $\sqrt{\sigma} \leq \irr (G) \leq \sqrt{m\sigma}$. Then, from~\eqref{f1boundSigman4} we have
\begin{equation}~\label{eq3boundSigman02}
\sigma_{\min}<\frac{n^2+n^2\delta+(\Delta+4)^2}{(\Delta^2-2)\sqrt{m(\Delta-1)}}.
\end{equation}
Therefore, the lower bound of $\sigma_{\min}$ from~\eqref{eq3boundSigman02} satisfy: 
\begin{equation}~\label{eq4boundSigman02}
\sigma_{\min}<\frac{n^2+n^2\delta+\Delta^2}{(\Delta^2-2)}.
\end{equation}
Since the term $n\Delta> \Delta-k$ where $0\leq k<n-2$. Thus, $\sigma_{\min}<(n-\Delta)^2$, then 
\begin{equation}~\label{eq5boundSigman02}
 \frac{(\Delta-2)^3}{n\Delta-\delta}\leq (n-\Delta)^2.
\end{equation}
From~\eqref{eq3boundSigman02}--\eqref{eq5boundSigman02} the bound~\eqref{eq1boundSigman02} is valid. Now, we need to confirm the lower bound~\eqref{eq2boundSigman02} of Sigma index. By considering~\eqref{f1boundSigman1} we have: 
\[
\sigma_{\min} \leq n \frac{mn(\Delta^2+2\Delta)+2m\delta\Delta}{(\Delta^2-2)(\Delta-1)^2+2n\Delta\sqrt{m\Delta^2}} (n+2) \log \binom{M_1(T)}{2n} \leq \sigma_{\max}.
\]
To rigorously prove the bound, we need to consider the tree that minimizes $\sigma$. The term $\frac{(\Delta - 2)^3}{n\Delta - \sigma}$ is small when $\Delta$ is close to $2$.  Since the path graph gives a loose bound, then
\begin{equation}~\label{eq6boundSigman02}
\sigma(T) = \sum_{i=1}^n d_i (n - 2i + 1).
\end{equation}
Thus, from~\eqref{eq6boundSigman02} we obtained on the lower bound as
\begin{equation}~\label{eq7boundalbertsonn02}
\sigma_{\min} \leq \frac{n\Delta^2}{\Delta+1}\leq \frac{nm\Delta^2}{n\Delta+\delta} .  
\end{equation}
So that, we already have $\sigma_{\min}>\Delta^2(\Delta-\delta)/6\delta(\Delta-1)$, then 
\[
\sigma_{\min}\leqslant\frac{\Delta^2 \left( \Delta^2 + \left(6 \delta^2 n + 1 - \delta \right) \Delta - 6 \delta^2 n - \delta \right)}{6 \delta \left(\Delta^2 - 1 \right)}.
\]
As desire
\end{proof}
\begin{proposition}~\label{lowerboundSigmastep2}
Let $T=(V,E)$ be a tree, where $n=|V|$ vertices and $m=|E|$ with maximum degree $\Delta$ and minmum degree $\delta$, then the upper bound of the minmum Sigma index $\sigma_{\min}$ is 
\begin{equation}~\label{eq1lowerboundSigmastep2}
\sigma_{\min}\geq \ceil{\log_2(n^2+1)-1}+2\Delta(T)-1.
\end{equation}
\end{proposition}
\begin{proof}
Let $T=(V,E)$ be a tree with $n>4$, we need to prove the upper bound for Sigma index by considering  the maximum level of any vertex of the tree is $\ell$. Actually, from~\eqref{eq1lowerboundAlbertsonstep2} when $n>2$ we find 
\[
\irr_{\min}\geq \ceil{\log_2(n+1)-1}+2\Delta(T)-1.
\]
Then, when $n\in [10,50]$ the bound~\eqref{eq1lowerboundSigmastep2} satisfy $\sigma_{\min}>35$. Assume $n>50$, in this case we have the bound
\begin{equation}~\label{eq2lowerboundSigmastep2}
\ceil{\log_2(n+1)-1} \leq  \ceil{\log_2(n^2+1)-1}+2\Delta(T)+1.
\end{equation}
Consequently, $\sigma_{\min}>\irr_{\min}$, thus from~\eqref{eq2lowerboundSigmastep2} the bound~\eqref{eq1lowerboundSigmastep2} is valid. As desire.
\end{proof}

\begin{lemma}~\label{lemmaboundSigman1}
Let $T_1,T_2$ be a tow trees, where $T_1$ have $\mathscr{D}=(d_1,\dots, d_i)$ an increasing degree sequence occur of $2$. Tree $T_2$ have $\mathscr{B}=(b_1,\dots, b_i)$ an increasing degree sequence of prime number, consider $n_1=|V(T_1)|,n_2=|V(T_2)|$ and $m_1=|E(T_1)|,m_2=|E(T_2)|$. Then, $n_2=n_1-1$, $m_2=m_1-1$ and $\Delta(T_2)=\Delta(T_1)+1$. Then the bound of Sigma index given by: 
\begin{equation}~\label{eq1lemmaboundSigman1}
\sigma_{\min}\leq  n_1(\Delta_1-1)^2.\frac{n_1\Delta_2^3+n_2\Delta_1^4+m_1\Delta_2^2}{\Delta_1(\Delta_2-1)^2} \leq \sigma_{\max}.
\end{equation}
\end{lemma} 
\begin{proof}
Let $T_1,T_2$ be a tow trees, then denote by $\irr_{\min}=\irr_{\min}(T_1)+\irr_{\min}(T_2)$ and $\irr_{\max}=\irr_{\max}(T_1)+\irr_{\max}(T_2)$. Since most of the properties that apply to the Albertson index can also be applied to the Sigma index, especially with regard to the bounds, taking into account the conditions given in each case. Thus, from Proposition~\ref{boundSigman01} according to~\eqref{f1boundSigman1} we find that 
\[
\sigma_{\min} \leq n \frac{mn(\Delta^2+2\Delta)+2m\delta\Delta}{(\Delta^2-2)(\Delta-1)^2+2n\Delta\sqrt{m\Delta^2}} (n+2) \log \binom{M_1(T)}{2n} \leq \sigma_{\max}.
\]
This suggests a general bound for the Sigma index, involving the number of vertices $n$, edges $m$, maximum degree $\Delta$, minimum degree $\delta$, possibly the sum of squared degrees. For two trees $T_1$ and $T_2$, we interpret $n = n_1 + n_2$, $m = m_1 + m_2$, and possibly $\Delta = \max(\Delta_1, \Delta_2)$, $\delta = \min(\delta_1, \delta_2)$. Then, for the upper bound~\eqref{eq2lemmaboundSigman1} it is clear $(\Delta_1^2-2)(\Delta_2-1)^2\geq  n_1(\Delta_1-1)^2$, by considering the Sigma index may depend only on the number of vertices or edges, and the maximum degree $\Delta$ satisfy   $n_1\Delta_2^3+n_2\Delta_1^4\leq (\Delta_1^2-2)(\Delta_2-1)^2$. Then, the upper bound of the minmum Sigma index is 
\begin{equation}~\label{eq2lemmaboundSigman1}
\sigma_{\min}\geq\frac{n_1\Delta_2^3+n_2\Delta_1^4+m_1\Delta_2^2}{\Delta_1(\Delta_2-1)^2}.
\end{equation}
The bound~\eqref{eq2lemmaboundSigman1} suggests the Sigma index grows with the number of vertices, the square and cube of the maximum degree, and the number of edges. The bound~\eqref{eq2lemmaboundSigman1} holds under the given condition $n_1\Delta_2^3+n_2\Delta_1^4\leq (\Delta_1^2-2)(\Delta_2-1)^2$.

Obviously, we can considering $\sigma(T)\geq n\Delta^2(\Delta-1)+m$, then the bound~\eqref{eq2lemmaboundSigman1} in the statement gives the right result. The lower bound of the minmum Sigma index is: 
    \begin{equation}~\label{eq3lemmaboundSigman1}
\sigma_{\min}\leqslant \frac{n_1\delta}{\Delta_1(\Delta_2-1)^2}n_1(\Delta_2-1)^3.
\end{equation}
Finally, from~\eqref{eq2lemmaboundSigman1} and \eqref{eq3lemmaboundSigman1} the bound~\eqref{eq1lemmaboundSigman1} is valid.
\end{proof}

\begin{lemma}~\label{lemmaboundSigman01}
Let $T$ be a tree with the maximum of Sigma index $\sigma_{\max}$ and the minmum Sigma index $\sigma_{\min}$. Let $\mathscr{D}=(d_1,\dots,d_n)$ be an increasing degree sequence, where $n=|V(T)|$ and $m=|E(T)|$, then 
\begin{equation}~\label{eq1lemmaboundSigman01}
\begin{cases}
\sigma_{\max} <\Delta(T)\dfrac{nm-\sigma_{\min}}{4},\\
\sigma_{\min} <\Delta(T)\dfrac{nm-\sigma_{\max}}{4}.
\end{cases}  
\end{equation}
\end{lemma}
\begin{proof}
Let $T$ be a tree of order at least 3, with the maximum of Sigma index $\sigma_{\max}$ and the minmum Sigma index $\sigma_{\min}$. Let $\mathscr{D}=(d_1,\dots,d_n)$ be a degree sequence where $d_n\geqslant \dots \geqslant d_1$. According to Lemma~\ref{lemmaboundAlbertsonn01} we find in~\eqref{eq1lemmaboundAlbertsonn01} for Albertson index
\[
\begin{cases}
\irr_{\max} <\Delta(T)(\irr_{\min}-n),\\
\irr_{\min} <\Delta(T)(\irr_{\max}-n).
\end{cases}  
\]
Thus, the lower bound of edges satisfy $m\leq 3(\sigma-n)$ and the lower bound of vertices satisfy $n\leq 3(\sigma-m)$, we need prove the bound~\eqref{eq1lemmaboundSigman01} for the maximum degree $\Delta$ of $T$ where $\Delta(T)>3$, let $\mathscr{D}=(d_{i1}, d_{i2}, \dots, d_{ik}, d_{j1}, d_{j2},\dots, d_{j\ell})$ be degree sequence of $\mathscr{D}$ as $\mathscr{D}_1=(d_{i1}, d_{i2}, \dots, d_{ik})$ and $\mathscr{D}_2=(d_{j1}, d_{j2},\dots, d_{j\ell})$ where $\{k,\ell\}\in \mathbb{N}$. Then, by considering $i\leq s \leqslant r \leq j$ we have 
\begin{equation}~\label{eq2lemmaboundSigman01}
\begin{cases}
    \irr_{\max} <\Delta(T)\left(\irr_{\min}-( \sum_{s=1}^{k}d_{is}+\sum_{r=1}^{\ell}d_{jr}\right),\\
\irr_{\min} <\Delta(T)\left(\irr_{\max}-( \sum_{s=1}^{k}d_{is}+\sum_{r=1}^{\ell}d_{jr}\right).
\end{cases}
\end{equation}
Thus, from~\eqref{eq2lemmaboundSigman01} noticed that $\sigma_{\max} <\Delta \floor{n/2}$ and $\sigma_{\min} <\Delta \floor{n/2}$, then $\sigma_{\max} <\Delta (\sigma_{\min}-\floor{n/2})$ and $\sigma_{\min} <\Delta (\sigma_{\max}-\floor{n/2})$. Thus, 
\begin{equation}~\label{eq3lemmaboundSigman01}
 n=   \sum_{s=1}^{k}d_{is}+\sum_{r=1}^{\ell}d_{jr}
\end{equation}
Therefore, in both cases $\lvert k-\ell\rvert=0$ or $\lvert k-\ell \rvert=1$, the bound~\eqref{eq1lemmaboundAlbertsonn01} is valid according to~\eqref{eq2lemmaboundAlbertsonn01}. Then, let $\eta\geqslant m$ satisfy 
\begin{equation}~\label{eq4lemmaboundSigman01}
    \begin{cases}
\sigma_{\max} < \Delta(T)(\sigma_{\min} - \eta), \\
\sigma_{\min} < \Delta(T)(\sigma_{\max} - \eta).
\end{cases}
\end{equation}
This term evaluates each intermediate degree of $\mathscr{D}_1$ and $\mathscr{D}_2$. So that, if $\eta=n$ from~\eqref{eq4lemmaboundSigman01} the bound~\eqref{eq1lemmaboundAlbertsonn01} is valid  for $\mathscr{D}$. As desire.
\end{proof}
\begin{lemma}~\label{lemmaboundSigman2}
Let $T$ be a tree with $n=|V(T)|,m=|E(T)$  with the maximum degree $\Delta$ and the minimum degree $\delta\geqslant 1$ and $0\leq \alpha \leq 1$. Then, the lower bound of Sigma index is 
 \[\sigma(T) \leq \floor*{\frac{3n^4-2mn}{\Delta}}2^{\alpha}+\Delta^2-n\delta.\]
\end{lemma}
\begin{proof}
Let $T$ be a tree with $\Delta\geqslant 3$, then we find that $\sigma(T)\leq 3(\Delta^2-n)^2$ by consideration $\Delta <n\delta$. This measures the sum of the squared differences for Sigma index of the degrees of adjacent vertices. Then, the lower bound given by: 
\begin{equation}~\label{eq1lemmaboundSigman2}
\sigma(T)\geq \floor*{\dfrac{3n^4-2mn}{\Delta}}.
\end{equation}
\case{1} If $T$ is a star tree $\mathcal{S}_n$, then the central vertex has degree $n-1$. By using the definition of Sigma index we have $\sigma(\mathcal{S}_n)=(n-1)(n-2)^2$. Then, 
\begin{equation}~\label{eq2lemmaboundSigman2}
\sigma(\mathcal{S}_n)\leq \floor*{\frac{3n^4-2mn}{\Delta}}2^{\alpha}+\Delta^2-n\delta.
\end{equation}
\case{2} If $T$ is a path $\mathcal{P}_n$ has vertices labeled 
$\{v_1,\dots,v_n\}$, aligns with a path graph in the context of trees,  we find in Lemma~\ref{lemmaboundAlbertsonn2} for Albertson index $\irr(\mathcal{P}_n)=2$, if $\mathcal{P}_n$ has a degree sequence $\mathscr{D}=(d_1,\dots,d_n)$. Thus, let $\alpha=\dfrac{n}{\Delta^2-1}$, then for the term $2^{n}$ is growth of the bound~\eqref{eq1lemmaboundSigman2}. So that, for prove the bound
\begin{equation}~\label{eq3lemmaboundSigman2}
\sigma(\mathcal{P}_n)\leq \floor*{\frac{3n^4-2mn}{\Delta}}2^{\alpha}+\Delta^2-n\delta.
\end{equation}
We need to determine if 
\begin{equation}~\label{eq4lemmaboundSigman2}
\floor*{\dfrac{3n^4-2mn}{\Delta}}-\floor*{\frac{3n^4-10n+m}{\Delta}}  
\end{equation}
satisfy the required related relationship. Let’s define:
\[
x = \dfrac{3n^4 - 2mn}{\Delta}, \quad y = \dfrac{3n^4 - 10n + m}{\Delta}.
\]
Then, for the bound~\eqref{eq4lemmaboundSigman2} for $\floor{x}$ and $\floor{y}$ we have
\begin{align*}
    x-y &=\dfrac{3n^4 - 2mn}{\Delta} - \dfrac{3n^4 - 10n + m}{\Delta}\\
    &= \dfrac{(3n^4 - 2mn) - (3n^4 - 10n + m)}{\Delta}\\
    &= \dfrac{-2mn + 10n - m}{\Delta} \\
    &= \dfrac{(10 - 2m)n - m}{\Delta}.
\end{align*}
Thus, we have $y=a-\lambda$, where
\[
\lambda=\dfrac{2(m - 5)n + m}{\Delta},
\]
To prove the bound~\eqref{eq4lemmaboundSigman2} we need to show $\floor{y}=\floor{x-\lambda}$, let $k>0$ an integer number, so $\{x\}=x-\floor{x}$, then $\floor{x}-\floor{x-k}=k$. We provided $\floor{x} = \floor{x - k} + k$, which holds when $ x - k \geq \floor{x - k}$. Thus, $\floor{x} - \floor{x-\lambda} = \floor{\lambda + \{a\}}$. In this case $(2m - 10)n + m = k \Delta$, then $x-y=-5/\Delta$ by considering $5/\Delta <1$. So that, by considering the term $0\leq \alpha \leq 1$ we have
\[
\floor{x-5/\Delta} - \floor{x} =
\begin{cases} 
0 & \text{if } \{x\} < 5/\Delta, \\
1 & \text{if } \{x\} \geq 5/\Delta.
\end{cases}
\]
Therefore,  
\[
\floor*{\dfrac{3n^4-2mn}{\Delta}}-\floor*{\frac{3n^4-10n+m}{\Delta}} >0.
\]
Finally, we find from all cases which had discussed we find that the bound~\eqref{eq3lemmaboundSigman2} is valid. As desire.
\end{proof}

\begin{proposition}~\label{boundstep2Sigman1}
Let $T$ be a tree, $M_2(T)$ be the second Zagreb index and $F(T)$ the forgotten index. Then, the upper bound of Sigma index is: 
\begin{equation}~\label{eq1boundstep2Sigman1}
\sigma(T)\geq  \sqrt{\frac{5F(T)+4M_2(T)-n\Delta^2}{\Delta(\Delta-1)}}.
\end{equation}
\end{proposition}
\begin{proof}
Let $\mathscr{D}=(d_1,\dots,d_n)$ be a degree sequence where $\sum_{i=1}d_i\leqslant n$ with $M_2(T)$ be the second Zagreb index and $F(T)$ the forgotten index. Since the required bound~\eqref{eq1boundstep2Sigman1} satisfy 
 $\sqrt{\sigma}=\sqrt{F-2M_2}$ by considering $\irr \leqslant \sqrt{\sigma}$ holds for the constant term $\Delta(\Delta-1)$ the bound
\begin{equation}~\label{eq2boundstep2Sigman1}
\sigma(T)\geq \frac{5F(T)+4M_2(T)}{\Delta-2}.
\end{equation}
Therefore, according to Proposition~\ref{boundstep2Albertsonn1} we find for Albertson index in~\eqref{eq4boundstep2Albertsonn1} to determine whether the required bound~\eqref{eq1boundstep2Sigman1} is achieved when the terms are satisfied with 
\[
\frac{F(T)+2M_2(T)}{\Delta}- \frac{F(T)+2M_2(T)-n\Delta}{\Delta(\Delta-1)}>0.
\]
In this case, we need to show the bound~\eqref{eq3boundstep2Sigman1} recognize that the problem in bound~\eqref{eq1boundstep2Sigman1} might imply a specific identity. By considering $F(T)-(\Delta-6)M_2(T)\leq n\Delta^2$ insight is that raising by considered the constant terms $\Delta-6, \Delta(\Delta-1)$, we compute 
\begin{equation}~\label{eq3boundstep2Sigman1}
\frac{5F(T)+4M_2(T)}{\Delta-2}-\frac{5F(T)+4M_2(T)-n\Delta^2}{\Delta(\Delta-1)},
\end{equation}
satisfy the term required. Thus, further simplification with $\Delta\geq 3$ and $\Delta>6$ for $\Delta-6, \Delta(\Delta-1)$  to evaluate the bound~\eqref{eq3boundstep2Sigman1} denote by $\irr_\mathscr{A}(T)$   as
\begin{align*}
\irr_\mathscr{A}(T)&=\frac{(5F(T) + 4M_2(T)) \cdot \Delta(\Delta - 1) - (5F(T) + 4M_2(T) - n\Delta^2) \cdot (\Delta - 2)}{\Delta(\Delta - 1)(\Delta - 2)}\\
&=\frac{n\Delta^3 + (5F(T) + 4M_2(T) - 2n) \Delta^2 - 10F(T) \Delta + 10F(T) + 8M_2(T)}{\Delta(\Delta - 1)(\Delta - 2)}>0.
\end{align*}
As desire.
\end{proof}

\begin{lemma}~\label{lemmastep2Sigman1}
Let $T \in \mathcal{T}_{n, \Delta}$. Then,let $v_0 \in V(T)$ be a vertex with maximum degree $\Delta$. For any vertex $v_{\ell}$ in $T$, different from $v_0$, where $\deg(v_{\ell})\geqslant 3$ satisfy 
\[
(3\delta)^\lambda \leq \sigma(T)\leq (\Delta-1)(n-2)^\lambda.
\]
The lower bound holds if and only if $T\cong \mathcal{P}_n$. The upper bound holds if and only if $T\cong K_{1,n-1}$.
\end{lemma}
\begin{proof}
According to Lemma~\ref{lemmastep2Albertsonn1} we find that for $v_0 \in V(T)$ be a vertex with maximum degree $\Delta$. A vertex $v_{\ell}$ in $T$, different from $v_0$, where $\deg(v_{\ell})=\lambda\geqslant 3$ satisfy in~\eqref{eq1lemmastep2Albertsonn1} as 
\[
\irr(T)\leq (n-1)(n-2)^\lambda.
\]
Thus, if $T\cong K_{1,n-1}$ we compute $\sigma(\mathcal{S}_n)=(n-1)(n-2)^2$. Since $\lambda\geqslant 3 >2$ satisfy $\sigma(\mathcal{S}_n)\leq (n-1)(n-2)^\lambda=\sigma(K_{1,n-1})$. 

Therefore, since $\irr(T)\leqslant\sqrt{\sigma(T)}$. Then, when $1\leq \delta\leq \Delta-2$ satisfy $\sigma(T)\geq (2\delta)^\lambda\geq 2^\lambda =\sigma(\mathcal{P}_n)$.
\end{proof}

\section{Conclusion}\label{sec5}
Through this paper, we presented a study of bounds on the topological indices which we elected to the Sigma Index and the Albertson Index.  According to a degree sequence $\mathscr{D}=(d_1,\dots,d_n)$, we provide this bounds with the maximum and minimum of both tropological indices where we find that for Sigma index in~\eqref{f1boundSigman1} as: 
\[
\sigma_{\min} \leq n \frac{mn(\Delta^2+2\Delta)+2m\delta\Delta}{(\Delta^2-2)(\Delta-1)^2+2n\Delta\sqrt{m\Delta^2}} (n+2) \log \binom{M_1(T)}{2n} \leq \sigma_{\max}.
\]
The lower bound holds in Lemma~\ref{lemmastep2Sigman1} According to Lemma~\ref{lemmastep2Albertsonn1} if $T\cong \mathcal{P}_n$.Also, for the upper bound holds if and only if $T\cong K_{1,n-1}$. The results are illustrated with examples, including a tree with a Fibonacci degree sequence and a pair of trees satisfying specific degree constraints. These findings contribute to the understanding of topological irregularities in graph theory, with potential applications in network analysis and molecular chemistry.

\section*{Declarations}
\begin{itemize}
	\item Funding: Not Funding.
	\item Conflict of interest/Competing interests: The author declare that there are no conflicts of interest or competing interests related to this study.
	\item Ethics approval and consent to participate: The author contributed equally to this work.
	\item Data availability statement: All data is included within the manuscript.
\end{itemize}

\end{document}